\documentclass[12pt,oneside,english]{amsart}
\usepackage[T1]{fontenc}
\usepackage[latin9]{inputenc}
\usepackage{mathrsfs}
\usepackage{amstext}
\usepackage{amsthm}
\usepackage{amssymb}

\makeatletter
\numberwithin{equation}{section}
\numberwithin{figure}{section}
\theoremstyle{plain}
\newtheorem{thm}{\protect\theoremname}
\theoremstyle{plain}
\newtheorem{prop}[thm]{\protect\propositionname}
\theoremstyle{plain}
\newtheorem{lem}[thm]{\protect\lemmaname}
\theoremstyle{remark}
\newtheorem{rem}[thm]{\protect\remarkname}

\makeatother

\usepackage{babel}
\providecommand{\lemmaname}{Lemma}
\providecommand{\propositionname}{Proposition}
\providecommand{\remarkname}{Remark}
\providecommand{\theoremname}{Theorem}

\begin{document}
\title{On Franke's theorem in the simplest case}
\author{Devadatta G. Hegde}
\begin{abstract}
For level one spherical automorphic forms on the upper half-plane,
we prove directly that every automorphic form is a sum of a cusp form
and a linear combination of Laurent coefficients of the standard Eisenstein
series. This is the simplest instance of Franke's general theorem,
which asserts that automorphic forms on a reductive group are spanned
by Laurent coefficients of Eisenstein series induced from cuspidal
automorphic forms on Levi subgroups. Unlike Franke's general argument,
ours does not invoke Langlands' construction of the discrete automorphic
spectrum from cuspidal Eisenstein series. It rests instead on basic
analytic properties of automorphic forms and Green's identity.
\end{abstract}

\maketitle
\tableofcontents{}

\section{Introduction}

\subsection{Orientation}

Langlands\textquoteright{} foundational work \cite{Langlands-Spectral}
on the spectral decomposition of automorphic forms may be divided
into two parts. The first part, covering Chapters I-VI, received an
exposition from Harish-Chandra \cite{HarishChandraBook} and concerns
the basic analytic properties of automorphic forms in general, the
Maass-Selberg relations, and the meromorphic continuation of cuspidal
Eisenstein series. Harish-Chandra does not treat ``the last and the
most difficult part (Chapter VII)'' of Langlands' work, which uses
a sophisticated residue scheme. 

Franke's argument \cite{Franke_firstPaper} proves that automorphic
forms on a reductive group are linear combinations of Laurent coefficients
of Eisenstein series induced from cusp forms on Levi subgroups. Franke's
proof rests on Chapter VII of Langlands\textquoteright{} work. Waldspurger
emphasizes the delicacy of the argument in his Bourbaki exposition
\cite{Waldspurger1997Franke} and the reader may consult its \S3
for a precise statement of Franke's theorem. 

The present paper grew out of an attempt to rethink this circle of
ideas \cite{Dattu-Rethinking_Langlands}. The guiding view is that
Franke\textquoteright s theorem should be regarded more in the spirit
of Chapters I-VI of \cite{Langlands-Spectral} and has nothing to
do with the residue calculus that is central to Langlands' approach
to \emph{spectral} \emph{decomposition}. This paper supports that
view by giving, in the simplest nontrivial case, a proof that uses
only the general analytic properties of automorphic forms together
with an automorphic Green's identity. The latter is Maass' version
of the Maass-Selberg relation, suitably extended to generalized eigenfunctions. 

\subsection{Statement of the theorem}

\subsubsection*{Notation}

Let
\[
\Gamma=SL(2,\mathbb{Z}),\quad\mathfrak{H}=\left\{ z\in\mathbb{C}:\Im(z)>0\right\} ,\quad X=\Gamma\backslash\mathfrak{H},
\]
\[
\Delta=-y^{2}(\partial_{x}^{2}+\partial_{y}^{2}),\quad d\mu=\frac{dxdy}{y^{2}}.
\]
Let 
\[
\mathscr{F}\subset\left\{ z\in\mathfrak{H}:|\Re(z)|\le1/2,\ |z|\ge1\right\} 
\]
be the standard fundamental domain for the action of $\Gamma$ on
$\mathfrak{H}$. 

\subsubsection*{Eisenstein series}

Let 
\[
Q=\left\{ \begin{pmatrix}* & *\\
0 & *
\end{pmatrix}\in SL(2,\mathbb{R})\right\} 
\]
be the standard maximal parabolic subgroup of $SL(2,\mathbb{R})$.
The standard \emph{Eisenstein series }is
\[
E(s;z):=\sum_{\gamma\in(\Gamma\cap Q)\backslash\Gamma}\text{Im}(\gamma z)^{\frac{1+s}{2}},\quad\Re(s)>1
\]
The map 
\[
s\mapsto E(s):=E(s;\bullet)\in C^{\infty}(X)
\]
has a vector-valued meromorphic continuation to $\mathbb{C}$ and
satisfies the functional equation 
\[
E(s)=c(s)E(-s)
\]
for a meromorphic complex-valued function $c(s)$. 

\subsubsection*{Automorphic forms}

The space of automorphic forms $\mathcal{A}(X)$ consists of $F\in C^{\infty}(X)$
satisfying
\begin{itemize}
\item ($\mathfrak{z}$-finite) $\exists$ $0\neq P\in\mathbb{C}[x]$ such
that $P(\Delta)F=0$. 
\item (uniform moderate growth) $\exists N\in\mathbb{Z}_{\ge0}$ such that
for all $a,b\in\mathbb{Z}_{\ge0}$,
\[
|\partial_{x}^{a}\partial_{y}^{b}F(x+iy)|\ll_{a,b}y^{N}\quad\text{on }\mathscr{F}.
\]
 
\end{itemize}
Eisenstein series provide examples of automorphic forms. At each $s_{0}\in\mathbb{C}$,
we have the Laurent expansion
\[
E(s)=\sum_{n=-\infty}^{\infty}E_{s_{0},n}(s-s_{0})^{n},\quad E_{s_{0},n}\in\mathcal{A}(X)
\]
Let 
\[
\mathcal{E}(X):=\text{span}\left\{ E_{s,n}:n\in\mathbb{Z}\text{ and }s\in\mathbb{C}\right\} 
\]
One has 
\[
\mathcal{E}(X)\subset\mathcal{A}(X).
\]

\subsubsection*{Cusp forms}

Let
\[
\mathcal{A}(0,\infty):=\text{span}\left\{ \partial_{s}^{n}(y^{s})\in C^{\infty}(0,\infty):s\in\mathbb{C},\ n\in\mathbb{Z}_{\ge0}\right\} 
\]
Write
\[
C:\mathcal{A}(X)\to\mathcal{A}(0,\infty),\quad CF(y):=\int_{-1/2}^{1/2}F(x+iy)dx
\]
for the constant term map. The space of cusp forms is
\[
\mathcal{A}_{0}(X):=\ker C
\]
For $F\in\mathcal{A}(X)$, one has 
\[
C(P(\Delta)F)=P(D)(CF),\quad D=-y^{2}\frac{d^{2}}{dy^{2}},
\]
and it follows that $CF\in\mathcal{A}(0,\infty)$.

\subsubsection*{Franke's theorem}
\begin{thm}
We have 
\[
\mathcal{A}(X)=\mathcal{A}_{0}(X)\oplus\mathcal{E}(X)
\]
\end{thm}

\subsection{The issue to be addressed}

First we need

\subsubsection*{Some standard facts about automorphic forms}

A standard reference we use is Borel's book \cite{Borel_SL2}.
\begin{itemize}
\item (\cite{Borel_SL2}, Theorem 7.5) For every $F\in\mathcal{A}(X)$,
one has
\[
|F(z)-CF(\Im z)|\ll_{N}\Im(z)^{-N},\quad z\in\mathscr{F},\ N\in\mathbb{Z}.
\]
In particular, if $\varphi\in\mathcal{A}_{0}(X)$, then $\varphi$
is rapidly decreasing on $\mathscr{F}$ (\cite{Borel_SL2}, Theorem
7.5) and the pairing
\[
\langle\varphi,F\rangle=\int_{X}\varphi\overline{F}d\mu,\quad F\in\mathcal{A}(X),\ \varphi\in\mathcal{A}_{0}(X)
\]
is well defined. 
\item Define
\[
\mathcal{A}_{0}(X)^{\perp}:=\left\{ F\in\mathcal{A}(X):\langle\varphi,F\rangle=0,\ \forall\varphi\in\mathcal{A}_{0}(X)\right\} 
\]
The map $C:\mathcal{A}_{0}(X)^{\perp}\to\mathcal{A}(0,\infty)$ is
\emph{injective}. Proof: If $F\in\mathcal{A}_{0}(X)^{\perp}$ and
$CF=0$, then 
\[
\langle F,F\rangle=0\quad\implies\quad F=0.
\]
\item $\mathcal{E}(X)\subset\mathcal{A}_{0}(X)^{\perp}$. In the convergence
region of $E(s)$, a direct calculation gives
\[
\langle E(s),\varphi\rangle\equiv0,\quad\forall\varphi\in\mathcal{A}_{0}(X).
\]
By meromorphic continuation the relation holds for all $s\in\mathbb{C}$.
Extension to Laurent coefficients follows immediately. 
\end{itemize}
\begin{prop}
There is an algebraic direct-sum decomposition 
\[
\mathcal{A}(X)=\mathcal{A}_{0}(X)\oplus\mathcal{A}_{0}(X)^{\perp}.
\]
\end{prop}

\begin{proof}
We need the finite dimensionality result (\cite{Borel_SL2} theorem
8.5): if $P\in\mathbb{C}[x]$, then
\[
\mathcal{H}_{P}:=\left\{ F\in\mathcal{A}(X):P(\Delta)F=0\right\} ,
\]
is finite dimensional. 

Let $F\in\mathcal{A}(X)$ and choose $P\in\mathbb{C}[x]$ such that
$P(\Delta)F=0$. The space
\[
\mathcal{H}_{P,0}:=\mathcal{H}_{P}\cap\mathcal{A}_{0}(X)\subset L^{2}(X)
\]
has an orthonormal basis of $\Delta$-eigenfunctions. Choose such
a basis $\varphi_{1},\dots,\varphi_{n}$ and let 
\[
F_{\text{cusp}}=\sum_{j=1}^{n}\langle F,\varphi_{j}\rangle\varphi_{j}
\]

We claim 
\[
\langle F-F_{\text{cusp}},\varphi\rangle=0,\quad\forall\varphi\in\mathcal{A}_{0}(X)
\]
Since $\mathbb{C}[\Delta]\varphi\subset L^{2}(X)$ is finite dimensional,
it is enough to prove the above equality for a $\Delta$-eigenfunction.
Let
\[
\Delta\varphi=\mu\varphi,\quad\text{for some }\mu\in\mathbb{R}
\]
If $P(\mu)\neq0$, then 
\[
0=\langle P(\Delta)F,\varphi\rangle=\overline{P(\mu)}\langle F,\varphi\rangle\quad\implies\quad\langle F,\varphi\rangle=0
\]
If $P(\mu)=0$, then $\varphi\in\mathcal{H}_{P,0}$. 
\end{proof}

\subsubsection*{A reduction }

For $F\in\mathcal{A}(X)$, the space 
\[
V_{F}:=\left\{ P(\Delta)F:P\in\mathbb{C}[x]\right\} 
\]
is a finite dimensional $\mathbb{C}[x]$-module. By Jordan decomposition,
to prove Franke's theorem, it is enough to show
\[
\mathcal{A}_{\lambda,N}^{\perp}:=\left\{ F\in\mathcal{A}_{0}(X)^{\perp}:(\Delta-\lambda)^{N}F=0\right\} \subset\mathcal{E}(X)
\]
for $\lambda\in\mathbb{C}$ and $N\in\mathbb{Z}_{\ge1}$. The case
$N=0$ is vacuous. 

Let 
\[
\mathcal{E}_{\lambda,N}=\left\{ F\in\mathcal{E}(X):(\Delta-\lambda)^{N}F=0\right\} 
\]
By the injectivity of the constant term map on $\mathcal{A}_{0}(X)^{\perp}$,
it is enough to show that 
\[
C\left(\mathcal{A}_{\lambda,N}^{\perp}\right)=C\left(\mathcal{E}_{\lambda,N}\right).
\]

\subsubsection*{Dimension of $C\left(\mathcal{E}_{\lambda,N}\right)$}

A direct computation shows that
\[
\dim\mathcal{E}_{\lambda,N}=N.
\]
Since $C$ is injective on $\mathcal{A}_{0}(X)^{\perp}$, it follows
that
\[
\dim C\left(\mathcal{E}_{\lambda,N}\right)=N.
\]
Indeed, write 
\[
\lambda(s)=(1-s^{2})/4
\]
For $\lambda\neq\frac{1}{4}$, choose $s_{0}\neq0$ such that $\lambda(s_{0})=\lambda$.
By the functional equation of $E(s)$, the Laurent coefficients at
$s_{0}$ and $-s_{0}$ span the same generalized eigenspace. The first
$N$ non-zero Laurent coefficients at such an $s_{0}$ form a basis
for $\mathcal{E}_{\lambda,N}$. In the exceptional case $\lambda=\frac{1}{4}$
(i.e., $s_{0}=0$), one may take

\[
\left\{ E_{0,1},E_{0,3},\dots,E_{0,2N-1}\right\} 
\]
as a basis. The even coefficients are redundant:
\[
\text{span}\left\{ E_{0,2},E_{0,4},\dots,E_{0,2N}\right\} \subset\text{span}\left\{ E_{0,1},E_{0,3},\dots,E_{0,2N-1}\right\} .
\]

\subsubsection*{Dimension mismatch and the issue}

Let 
\[
\mathscr{W}_{\lambda,N}:=\left\{ f\in C^{\infty}(0,\infty):(D-\lambda)^{N}f=0\right\} .
\]
We have
\[
F\in\mathcal{A}_{\lambda,N}^{\perp}\quad\implies\quad CF\in\mathscr{W}_{\lambda,N}
\]
From the theory of ordinary differential equations, 
\[
\dim\mathscr{W}_{\lambda,N}=2N.
\]
Thus the differential equation alone leaves a space twice as large
as the space supplied by Eisenstein series. The point is to show that
automorphy cuts this space down by half. 

Let

\[
\mathscr{V}_{\lambda,N}:=\left\{ CF:F\in\mathcal{A}_{\lambda,N}^{\perp}\right\} 
\]
be the image of $\mathcal{A}_{\lambda,N}^{\perp}$ under the constant
term map. The proof reduces to controlling the dimension of this space.
We have 
\[
C\left(\mathcal{E}_{\lambda,N}\right)\subset\mathscr{V}_{\lambda,N}\quad\implies\quad\dim\mathscr{V}_{\lambda,N}\ge\dim C\left(\mathcal{E}_{\lambda,N}\right)=N
\]
The issue is to show
\[
\dim\mathscr{V}_{\lambda,N}\le N
\]
which is a consequence of the following. 
\begin{thm}
\label{thm:skew-form}There exists a non-degenerate skew-symmetric
bilinear form
\[
\Omega_{\lambda,N}:\mathscr{W}_{\lambda,N}\times\mathscr{W}_{\lambda,N}\to\mathbb{C}
\]
for which $\mathscr{V}_{\lambda,N}$ is an isotropic subspace. Consequently,
together with $\dim\mathscr{V}_{\lambda,N}\ge N$ above,
\[
\dim\mathscr{V}_{\lambda,N}=N,
\]
and $\mathscr{V}_{\lambda,N}$ is a maximal isotropic subspace for
$\Omega_{\lambda,N}$. 
\end{thm}

This theorem will be proved in the next section. 

\section{Proof of the theorem}

The proof uses Maass' version \cite{MaassWaveForms} of the Maass-Selberg
relations, suitably extended for generalized $\Delta$-eigenfunctions.
It is also the version Harish-Chandra used in his work (see Varadarajan
\cite{Varadarajan_89book}, \S8.6). 

For $T>1$ let 
\[
X^{T}=\left\{ z\in\mathscr{F}:\Im(z)\le T\right\} 
\]
be the truncation of $X$ at height $T$. We think of $X^{T}$ as
the truncation of the modular curve and write 
\[
\partial X^{T}:=\left\{ x+iy:-1/2\le x\le1/2,\ y=T\right\} 
\]

Throughout this section, let 
\[
U,V\in\mathcal{A}(X)\qquad u:=CU,\quad v:=CV.
\]
For $\lambda\in\mathbb{C}$ set 
\[
L_{\lambda}:=\Delta-\lambda,\quad\ell_{\lambda}=D-\lambda,\quad D=-y^{2}\frac{d^{2}}{dy^{2}}.
\]
It follows immediately that
\[
L_{\lambda}^{N}U=0\quad\implies\quad\ell_{\lambda}^{N}u=0.
\]

Let 
\[
\mathcal{B}_{T}(U,V):=\int_{-1/2}^{1/2}\left(V\partial_{y}U-U\partial_{y}V\right)(x+iT)dx
\]
The following identities are standard (Robert \cite{Robert_Multi},
\S14.3 and Iwaniec \cite{IwaniecBookSpectral}, Theorem 6.14).
\begin{lem}
\label{lem:Green-identities}(Green identities)

(a) Simple Green's identity:
\[
\int_{X^{T}}\left(U\Delta V-V\Delta U\right)d\mu=\int_{X^{T}}\left(UL_{\lambda}V-VL_{\lambda}U\right)d\mu=\mathcal{B}_{T}(U,V).
\]

(b) Higher Green's identity: for every $m\in\mathbb{Z}_{\ge0}$,
\[
\int_{X^{T}}\left(UL_{\lambda}^{m+1}V-VL_{\lambda}^{m+1}U\right)d\mu=\sum_{j=0}^{m}\mathcal{B}_{T}(L_{\lambda}^{j}U,L_{\lambda}^{m-j}V)
\]
\end{lem}

\begin{proof}
Part (b) follows from part (a) by a telescoping cancellation. 
\end{proof}
\bigskip{}

For $a,b\in C^{\infty}(0,\infty)$, let
\[
B_{T}(a,b):=a'(T)b(T)-a(T)b'(T).
\]

\begin{lem}
\label{lem:passing-to-the-const}(passing to the constant terms) We
have 
\[
\mathcal{B}_{T}(U,V)=B_{T}(u,v)+R_{T}(U,V),
\]
and 
\[
\lim_{T\to\infty}R_{T}(U,V)=0
\]
\end{lem}

\begin{proof}
Write 
\[
U=u+U^{\circ},\qquad V=v+V^{\circ}.
\]
For all $y>0$, 
\[
\int_{-1/2}^{1/2}U^{\circ}(x+iy)dx=\int_{-1/2}^{1/2}V^{\circ}(x+iy)dx=0,
\]
and differentiating these identities with respect to $y$ gives
\[
\int_{-1/2}^{1/2}U_{y}^{\circ}(x+iy)dx=\int_{-1/2}^{1/2}V_{y}^{\circ}(x+iy)dx=0.
\]
Expanding $\mathcal{B}_{T}(U,V)$, all mixed terms vanish after integration
over $x$, so that 
\[
\mathcal{B}_{T}(U,V)=B_{T}(u,v)+\int_{-1/2}^{1/2}\left(V^{\circ}\partial_{y}U^{\circ}-U^{\circ}\partial_{y}V^{\circ}\right)(x+iT)dx
\]
Let
\[
R_{T}(U,V):=\int_{-1/2}^{1/2}\left(V^{\circ}\partial_{y}U^{\circ}-U^{\circ}\partial_{y}V^{\circ}\right)(x+iT)dx
\]
We have 
\[
\lim_{T\to\infty}R_{T}(U,V)=0
\]
since $U^{\circ}$ and $V^{\circ}$ are rapidly decreasing and their
first $y$-derivatives are of moderate growth as $T\to\infty$.
\end{proof}
The following lemmas \ref{lem:indep-of-T} and \ref{lem:non-degeneracy}
are quite standard from the point of view of ordinary differential
equations in the context of \emph{Lagrange's identity} (\cite{Coddington-Levinson},
\S3.6). For the reader's convenience, we present the proofs. 

For $a,b\in C^{\infty}(0,\infty)$, $N\in\mathbb{Z}_{\ge1}$, and
$T\ge1$, let
\[
\Omega_{\lambda,N;T}(a,b):=\sum_{j=0}^{N-1}B_{T}(\ell_{\lambda}^{j}a,\ell_{\lambda}^{N-1-j}b)
\]
This is the candidate symplectic form on 
\[
\mathscr{W}_{\lambda,N}:=\ker\ell_{\lambda}^{N}=\left\{ f\in C^{\infty}(0,\infty):\ell_{\lambda}^{N}f=0\right\} 
\]

\begin{lem}
\label{lem:indep-of-T}(independence of $T$) If $a,b\in\mathscr{W}_{\lambda,N}$,
then $\Omega_{\lambda,N;T}(a,b)$ is independent of $T$. 
\end{lem}

\begin{proof}
An easy calculation shows
\[
\frac{d}{dy}B_{y}(a,b)=a''b-ab''=\frac{a(\ell_{\lambda}b)-(\ell_{\lambda}a)b}{y^{2}}
\]
Thus, 
\[
\frac{d}{dy}\Omega_{\lambda,m+1;y}(a,b)=\frac{1}{y^{2}}\sum_{j=0}^{m}\left[(\ell_{\lambda}^{j}a)(\ell_{\lambda}^{m+1-j}b)-(\ell_{\lambda}^{j+1}a)(\ell_{\lambda}^{m-j}b)\right]
\]
\[
=\frac{a(\ell_{\lambda}^{m+1}b)-(\ell_{\lambda}^{m+1}a)b}{y^{2}}
\]
Taking $m+1=N$, the result follows. 
\end{proof}
We write $\Omega_{\lambda,N}$ for this skew form on $\mathscr{W}_{\lambda,N}$. 
\begin{lem}
\label{lem:non-degeneracy}The skew-form $\Omega_{\lambda,N}$ is
non-degenerate for all $\lambda\in\mathbb{C}$ and $N\in\mathbb{Z}_{\ge1}$. 
\end{lem}

\begin{proof}
For $a\in C^{\infty}(0,\infty)$, let 
\[
p_{j}(a)=(\ell_{\lambda}^{j}a)(T),\quad q_{j}(a)=(\ell_{\lambda}^{N-1-j}a)'(T),\quad j=0,\dots,N-1
\]
The existence and uniqueness theorem for initial value problems for
the ODEs, applied to $\ell_{\lambda}^{N}a=0$, implies that the map
\[
J_{T}:\mathscr{W}_{\lambda,N}\simeq\mathbb{C}^{2N},\quad a\mapsto\left(p_{0}(a),\dots,p_{N-1}(a),q_{0}(a),\dots,q_{N-1}(a)\right)
\]
is an isomorphism. In the coordinates on $\mathscr{W}_{\lambda,N}$
given by $J_{T}$, the form $\Omega_{\lambda,N}$ is given by the
matrix
\[
\begin{pmatrix}0 & -\mathbf{1}_{N}\\
\mathbf{1}_{N} & 0
\end{pmatrix},\quad\mathbf{1}_{N}\text{: identity matrix of size }N
\]
\end{proof}
\begin{prop}
Let $U,V\in\mathcal{A}(X)$ with $u:=CU$ and $v:=CV$. If
\[
L_{\lambda}^{N}U=L_{\lambda}^{N}V=0,
\]
then 
\[
\Omega_{\lambda,N}(u,v)=0
\]
\end{prop}

\begin{proof}
By Lemma \ref{lem:passing-to-the-const}, for each $j=0,1,\dots,N-1$
we have 
\[
\mathcal{B}_{T}(L_{\lambda}^{j}U,L_{\lambda}^{N-1-j}V)=B_{T}(\ell_{\lambda}^{j}u,\ell_{\lambda}^{N-1-j}v)+R_{T}(L_{\lambda}^{j}U,L_{\lambda}^{N-1-j}V),
\]
and 
\[
\lim_{T\to\infty}R_{T}(L_{\lambda}^{j}U,L_{\lambda}^{N-1-j}V)=0.
\]
Since $L_{\lambda}^{N}U=L_{\lambda}^{N}V=0$, the higher Green's identity
gives 
\[
0=\sum_{j=0}^{N-1}\mathcal{B}_{T}(L_{\lambda}^{j}U,L_{\lambda}^{N-1-j}V).
\]
Therefore, 
\[
0=\Omega_{\lambda,N;T}(u,v)+\sum_{j=0}^{N-1}R_{T}(L_{\lambda}^{j}U,L_{\lambda}^{N-1-j}V)
\]
Letting $T\to\infty$ and using Lemma \ref{lem:indep-of-T}, we obtain
\[
\Omega_{\lambda,N}(u,v)=0.
\]
\end{proof}
We summarize the proof of Franke's theorem. The above proposition
shows that $\mathscr{V}_{\lambda,N}$ is an isotropic subspace for
the non-degenerate symplectic form $\Omega_{\lambda,N}$ on $\mathscr{W}_{\lambda,N}$;
since $\dim\mathscr{W}_{\lambda,N}=2N$, any isotropic subspace has
dimension at most $N$. Combined with
\[
C\left(\mathcal{E}_{\lambda,N}\right)\subset\mathscr{V}_{\lambda,N}\quad\text{and}\quad\dim C\left(\mathcal{E}_{\lambda,N}\right)=N,
\]
and we obtain 
\[
\dim\mathscr{V}_{\lambda,N}=N\quad\text{and}\quad C\left(\mathcal{E}_{\lambda,N}\right)=\mathscr{V}_{\lambda,N}
\]
By the injectivity of $C$ on $\mathcal{A}_{0}(X)^{\perp}$, it follows
that 
\[
\mathcal{A}_{\lambda,N}^{\perp}=\mathcal{E}_{\lambda,N}
\]
for all $\lambda\in\mathbb{C}$ and $N\ge1$. Applying this to each
generalized eigenspace in the finite-dimensional $\mathbb{C}[\Delta]$-module
generated by $F\in\mathcal{A}_{0}(X)^{\perp}$ proves Franke's theorem. 

\subsection*{Remarks}
\begin{rem}
Franke's theorem, if proved in the above fashion in full generality,
would allow one to pass from the meromorphic continuation of cuspidal
Eisenstein series to general Eisenstein series, as remarked by Bernstein-Lapid
(\cite{Lapid-Perspectives-pub}, \S1). Since constant terms of cuspidal
Eisenstein series have a simpler structure than general Eisenstein
series, their meromorphic continuation is usually technically easier. 
\end{rem}

\begin{rem}
As a rough analogy, recall that for real reductive groups the classification
of irreducible \emph{admissible} representations is far better understood
than that of irreducible \emph{unitary} ones --- the unitary problem
being the deeper and harder of the two. A similar gap, in the author's
view, separates Franke's theorem from Langlands' contour deformation
method. In every case that is understood, the latter determines precisely
which Laurent coefficients of Eisenstein series yield square-integrable
automorphic forms, and hence generate unitary representations of the
adele group. For this reason Langlands' work is a far more serious
affair than Franke's theorem.
\end{rem}

\subsubsection*{Acknowledgements}

I wish to thank Dr. Paul Boisseau who pointed out the relevance of
\emph{Lagrange's identity }in the proof presented here. 

\bibliographystyle{plain}
\bibliography{references}

\end{document}